\documentclass{amsart}
\pdfoutput=1
\usepackage[utf8]{inputenc}
\usepackage[T1]{fontenc}

\usepackage{mathrsfs}
\usepackage{mathtools}
\usepackage{thmtools}

\usepackage[margin=1.5in]{geometry}
\usepackage{tikz,tikz-3dplot}
\usepackage{enumitem}

\usepackage[
backend=biber,
style=numeric,
sorting=ynt
]{biblatex}

\usepackage[%
	colorlinks,%
	bookmarksdepth=3,%
	citecolor=blue%
]{hyperref}

\usepackage[capitalise]{cleveref}

\theoremstyle{plain}
\newtheorem{thm}{Theorem}[section]
\newtheorem{conjecture}[thm]{Conjecture}
\newtheorem{prop}[thm]{Proposition}

\newtheorem{lemma}[thm]{Lemma}

\theoremstyle{definition}

\newtheorem{question}[thm]{Question}

\newtheorem{example}[thm]{Example}
\newtheorem{rem}[thm]{Remark}

\newtheorem{defin}[thm]{Definition}

\newcommand{\M}{\mathcal{M}}
\newcommand{\N}{\mathcal{N}}
\newcommand{\T}{\mathcal{T}}
\newcommand{\V}{\mathcal{V}}

\newcommand{\QQ}{\mathbb{Q}}
\newcommand{\CC}{\mathbb{C}}
\newcommand{\OO}{\mathcal{O}}
\newcommand{\PP}{\mathbb{P}}
\newcommand{\ZZ}{\mathbb{Z}}

\newcommand{\divv}{\operatorname{div}}
\newcommand{\rk}{\operatorname{rk}}
\newcommand{\Cl}{\operatorname{Cl}}
\newcommand{\D}{\mathcal{D}}
\newcommand{\Pic}{\operatorname{Pic}}
\newcommand{\Hom}{\operatorname{Hom}}
\newcommand{\cone}{\operatorname{cone}}
\newcommand{\ord}{\operatorname{ord}}
\newcommand{\NS}{\operatorname{NS}}

\begin{document}

\title{The Generalised Mukai Conjecture for Spherical Varieties}

\author[G.~Gagliardi]{Giuliano Gagliardi}
\address[G.~Gagliardi]{d-fine AG\\Brandschenkestrasse 150\\8002 Zurich\\Switzerland}
\email{}
\author[J.~Hofscheier]{Johannes Hofscheier}
\author[H.~Pearson]{Heath Pearson}
\address[J.~Hofscheier and H.~Pearson]{School of Mathematical Sciences\\University of Nottingham\\ Nottingham\\NG7 2RD\\UK}
\email{\{johannes.hofscheier, heath.pearson\}@nottingham.ac.uk}

\subjclass{Primary: 14M27; Secondary: 14J45, 14M25, 14E30}
\keywords{Complexity, Fano variety, Mukai conjecture, Spherical variety, Toric variety}

\begin{abstract}
We prove the generalised Mukai conjecture for \(\QQ\)-factorial spherical Fano varieties. In this case, a stronger inequality holds featuring an extra term—the minimum absolute complexity of a log Calabi-Yau pair—which measures how close the Fano variety is to being toric.
\end{abstract}

\maketitle{}

\section{Introduction}\noindent\phantomsection\label{sec:intro} In this paper we work over the field of complex numbers.

A classical question in algebraic geometry is the characterisation of projective space \(\PP^n\) among smooth Fano varieties.
To this end, many criteria are known, see for instance~\cite{KO73,Mo79,CMSB02}.
In 1988, Mukai~\cite[Conjecture~4]{Mukai88} extended this question by suggesting a characterisation of powers of projective space \({(\PP^n)}^r\) among smooth Fano varieties.
Further extending Mukai's work, Bonavero, Casagrande, Debarre, and Druel proposed the following conjecture in 2002:
\begin{conjecture}[{\cite[Conjecture]{BCDD}}, The generalised Mukai conjecture]\noindent\phantomsection\label{Mukai conjecture}
	Let \(X\) be a smooth Fano variety, then
	\[
		(\iota_X-1) \rho_X \leq \dim X,
	\]
	with equality if and only if \(X\cong{\left(\PP^{\iota_X-1}\right)}^{\rho_X}\).
\end{conjecture}
Here, \(\rho_X\coloneqq \rk(\NS(X))\) denotes the \emph{Picard number}, which is the rank of the \emph{N\'eron-Severi group \(\NS(X)\)} of \(X\), i.e.\ the group of Weil divisors on \(X\) modulo algebraic equivalence. And
\[
    \iota_X\coloneqq \min\{-K_X\cdot C\,|\, C\text{ rational irreducible curve in }X\}
\]
denotes the \emph{pseudo-index} of \(X\).

\Cref{Mukai conjecture} has been verified in various special cases, but remains open in general.
The known cases include: for dimensions \(\le5\)~\cite{BCDD, ACO04}; for Picard number one~\cite{Mukai88}; for pseudo-index \( \iota_X \ge \frac{\dim X + 3}3\)~\cite{NO10}; for toric varieties~\cite{C06}; for horospherical varieties~\cite{P10}; and for symmetric varieties~\cite{GH15}.

In this paper, we prove the \emph{generalised Mukai conjecture for spherical varieties}.
Recall that a \emph{spherical variety} is an irreducible normal variety \(X\), equipped with the action of a connected reductive group \(G\), such that \(X\) has an open \(B\)-orbit for a Borel subgroup \(B \subseteq G\).
In this setting, we show that a slightly stronger version of the generalised Mukai conjecture holds, \Cref{thm:Generalised Mukai conjecture}, which allows mild singularities. Recall that a projective variety is a \emph{Fano variety} if its anticanonical divisor class is ample.
Furthermore, a variety is called \emph{\(\mathbb{Q}\)-factorial} if a multiple of each Weil divisor is Cartier.
The aim of this paper is to prove the following theorem.
\begin{thm}[The spherical generalised  Mukai conjecture]\noindent\phantomsection\label{thm:Generalised Mukai conjecture}
	Let \(X\) be a \(\QQ\)-factorial spherical Fano variety, then
    \[
        (\iota_X-1)\rho_X \leq \dim X,
    \]
    with equality if and only if \(X\cong{(\PP^{\iota_X-1})}^{\rho_X}\).
\end{thm}
As mentioned above, \Cref{thm:Generalised Mukai conjecture} has previously been shown for \(\QQ\)-factorial Gorenstein Fano toric varieties~\cite{C06}, horospherical varieties~\cite{P10}, and symmetric varieties~\cite{GH15}.
The latter work is particularly relevant for this paper, as here, a combinatorial strategy is proposed to prove \Cref{thm:Generalised Mukai conjecture} for the whole class of spherical varieties --- which includes toric, horospherical, and symmetric varieties.
In particular, \cite[Definition~1.3]{GH15} introduces a combinatorial function \(\widetilde{\wp}\) which associates to an arbitrary spherical variety \(X\), a number \(\widetilde{\wp}(X) \in \QQ \cup \{ \infty \}\) (see \Cref{def: P}).
When \(X\) is a \(\QQ\)-factorial spherical Fano variety, this provides an upper bound for the product \((\iota_X-1)\rho_X\).
\begin{thm}[{\cite[Corollary~4.4]{GH15}}]\noindent\phantomsection\label{P function conjecture intro}
    Let \(X\) be a \(\QQ\)-factorial spherical Fano variety, then
    \[
        (\iota_X-1)\rho_X \le \dim X-\widetilde{\wp}(X).
    \]
\end{thm}

\begin{rem} In \cite{GH15}, it is assumed that \(X\) is \(\QQ\)-factorial and Gorenstein, that is \(-K_X\) is a Cartier divisor. This assumption was introduced to make use of Casagrande's proof of the Mukai conjecture for \(\QQ\)-factorial Gorenstein Fano toric varieties \cite{C06}. However, the arguments of \cite[Section~4]{GH15} do not make use of the Gorenstein assumption, which allows us to state \Cref{P function conjecture intro} in this generality.
\end{rem}

The core idea of this paper is to interpret \(\widetilde{\wp}(X)\) in terms of the absolute complexity of log canonical pairs \((X,D)\), and to use this to derive an upper bound for \(\widetilde{\wp}(X)\), and therefore for \((\iota_X-1)\rho_X\).
Recall that if \(X\) is complete, then a pair \((X,D)\) is called a \emph{log pair} if \(D\) is an effective \(\QQ\)-divisor such that \(K_X+D\) is \(\QQ\)-Cartier.
It is well-known that for each log pair \((X,D)\) there exists a \emph{log resolution} \(\pi \colon X' \to X\), i.e.\ \(\pi\) is a proper birational morphism such that \(X'\) is smooth, the exceptional locus of \(\pi\) is a divisor, and the union of the exceptional locus and the strict transform of the support of \(D\) is a simple normal crossings divisor (see \Cref{sec:constructing-lc-pair}).
A log pair \((X,D)\) is said to be \emph{log canonical} if for any log resolution of singularities \(\pi \colon X' \to X\), in the equation \(K_{X'}=\pi^*(K_X+D)-\pi_*^{-1}(D)+\sum_E a_E E\), each of the coefficients \(a_E\) satisfy \(a_E \ge -1\) (here \(E\subseteq X'\) runs through the components of the exceptional locus of \(\pi\)).

Following~\cite[Definition~1.7]{Geomcharoftoric}, the \emph{absolute complexity} of the log pair \((X,D)\) is \(\gamma(X,D) \coloneqq \dim X + \rho_X - d(D)\), where \(d(D)\) denotes the sum of the coefficients of the components of the Weil divisor \(D\), i.e.\ \(d(\sum_i a_i D_i)=\sum_i a_i\), where the \(D_i\) are prime divisors.
This quantity appears in the following theorem, which characterises toric varieties using the absolute complexity of a log canonical pair.

\begin{thm}[{\cite[Corollary~of~Theorem~1.2]{Geomcharoftoric}}]\noindent\phantomsection\label{complexity zero means toric}
    Let \(X\) be a complete variety, and let \((X,D)\) be a log canonical pair such that \(-(K_X+D)\) is nef.
    If \(\gamma\coloneqq \gamma(X,D) < 1\), then \(\gamma \geq 0\) and there is a divisor \(D'\) such that \((X,D')\) is a toric pair.
    Furthermore, \(D'\ge\lfloor D \rfloor\) and \(D'\) shares all but at most \(\lfloor 2\gamma\rfloor\) components with \(D\). 

\end{thm}

Recall that two \(\mathbb{Q}\)-divisors \(D_1\) and \(D_2\) in a normal variety \(X\) are said to be \emph{\(\mathbb{Q}\)-linearly equivalent}, denoted by \(D_1 \sim_{\mathbb{Q}} D_2\), if there exists an integer \(r\) such that \(rD_1\) and \(rD_2\) are \(\ZZ\)-divisors and linearly equivalent in the usual sense. 

A \emph{log Calabi-Yau pair} is a log canonical pair \((X,D)\) such that \(D\sim_\QQ -K_X\). Motivated by the previous theorem, for a complete variety \(X\) we introduce
\[
    \gamma(X) \coloneqq \inf\{\gamma(X,D) \mid (X,D)\text{ log Calabi-Yau pair}\}\ge0,
\]
which features in the following lower bound. Recall that a normal variety \(X\) is called \emph{\(\mathbb{Q}\)-Gorenstein} if a multiple of its anticanonical divisor is Cartier.

\begin{thm}[\(=\) \Cref{P function theorem}]\noindent\phantomsection\label{thm:key theorem}
    Let \(X\) be a complete \(\QQ\)-Gorenstein spherical variety, then \[\widetilde{\wp}(X)\ge\gamma(X),\] with \(\widetilde{\wp}(X)<1\) only if \(X\) is isomorphic to a toric variety.
\end{thm}

Combining \Cref{P function conjecture intro} and \Cref{thm:key theorem} gives the proof of \Cref{thm:Generalised Mukai conjecture}, the generalised Mukai conjecture for \(\QQ\)-factorial spherical varieties.

In \Cref{sec:examples}, there is an example to illustrate the proof.

\subsection{The Mukai-type conjecture}

Finally, in \Cref{sec:Mukai type} we remark on the \emph{Mukai-type conjecture}, posed by Gongyo in 2023, which extends the goal of the Mukai conjecture by conjecturally characterising products of projective spaces among the smooth Fano varieties. Here, we answer a question of Gongyo concerning this conjecture, and provide a succinct proof of the Mukai-type conjecture in the case of spherical varieties.


\section*{Acknowledgements}
The authors thank Kento Fujita for helpful comments.


\section{Defining the \texorpdfstring{\(\widetilde{\wp}\)}{p-tilde} function}\noindent\phantomsection\label{sec:background}
In this section, we breifly recall some of the combinatorial description of spherical varieties, which extends the toric dictionary. Then we define the \(\widetilde{\wp}\) function mentioned in the introduction.
We refer to~\cite{Timashev} and the references therein for a general introduction to spherical varieties.

In this paper, \(G\) denotes a connected reductive algebraic group, \(B\subseteq G\) a Borel subgroup, and \(X\) a spherical \(G\)-variety.
One assigns to \(X\) the lattice \(\M\subseteq \mathfrak{X}(B)\) of \(B\)-weights \(\chi\) of its \(B\)-semi-invariant rational functions, i.e.\ all functions \(f_\chi \in \CC(X)\) such that \(b\cdot f_\chi(x) = \chi(b) f(x)\) for all \(x\in X\) and \(b\in B\).
Since \(X\) has an open \(B\)-orbit, each \(B\)-semi-invariant rational function \(f_\chi\) is uniquely determined by its weight \(\chi \in \M\) (up to multiples in \(\CC^*\)).
Furthermore, we define the set \(\Delta\) of \(B\)-invariant prime divisors in \(X\).
Like in the toric dictionary, one assigns to every divisor \(D \in \Delta\) the dual vector \(\rho(D) \in \N \coloneqq \Hom(\M,\ZZ)\) defined by \(\langle \rho(D), \chi\rangle \coloneqq \nu_D(f_\chi)\), where \(\langle \cdot, \cdot\rangle \colon \N \times \M \to \ZZ\) denotes the dual pairing.
We denote by \(\rk X\) the rank of the lattice \(\M\).

If \(X\) is a spherical \(G\)-variety, then its open \(G\)-orbit is a homogeneous space \(G/H\) for a closed subgroup \(H \subseteq G\).
This homogeneous space has the distinguishing property that it contains an open \(B\)-orbit, and any such homogeneous space is called \emph{spherical}.
Like in the toric dictionary, the \(G\)-invariant divisors of a \(G\)-equivariant embedding \(G/H \hookrightarrow X\) correspond to rays \(\QQ_{\ge0}\rho(D)\) in the dual vector space \(\N_\QQ \coloneqq \N \otimes_\ZZ \QQ\).

Unlike toric geometry, rays corresponding to \(G\)-invariant prime divisors can only lie within a distinguished cone, the so-called \emph{valuation cone}:
\[
	\V \coloneqq \cone\left\{ \rho(D) \mid \text{\(D\) is a \(G\)-invariant prime divisor of a spherical embedding \(G/H\hookrightarrow X\)}\right\}.
\]

Extending the toric dictionary, the \(G\)-equivariant open embeddings of a spherical homogeneous space \(G/H\) into normal irreducible \(G\)-varieties, called \emph{spherical embeddings} of \(G/H\), are combinatorially described using the \emph{Luna-Vust theory} (see~\cite{Knop2012}).
According to this theory, there is a bijection between so-called \emph{colored fans} and isomorphism classes of spherical embeddings of \(G/H\). However, this characterisation will not play an important role in this paper.

Let \(G/H \hookrightarrow X\) be a spherical embedding. 
There is a finite set of \(B\)-invariant divisors in \(X\), called \emph{colors} and denoted by \(\D\), arising as the closures of \(B\)-invariant divisors in the spherical homogeneous space \(G/H\).

Throughout, we may assume that \(X\hookleftarrow G/H\) is a spherical variety with \(G=G^\mathrm{ss}\times S\), where \(G^\mathrm{ss}\) is semi-simple and simply connected, and \(S\) is a torus. All spherical varieties arise as a \(G\)-spherical embedding of such a \(G\), and moreover, under this assumption all line bundles on \(X\) are \(G\)-linearisable (see~\cite[Proposition 2.4 and the subsequent remark]{KKLV89}).

In~\cite[Proposition~4.1]{BrionAnticanonical}, Brion has shown that there is a natural representative of the anticanonical class \(-K_X\) in terms of the \(B\)-invariant prime divisors:
\[
	-K_X \coloneqq \sum_{D \in \Delta} m_D D.
\]
Here, the \(m_D\) are positive integers satisfying \(m_D=1\) if \(D\) is \(G\)-invariant, and \(m_D\ge1\) when \(D\) is a color.
In the latter case, there is a combinatorial formula for the coefficients.

To state the approach in~\cite{GH15}, we require the following intersection of half-spaces in \(\M_{\QQ} \coloneqq \M \otimes_{\ZZ} \QQ\):
\[
	Q^* \coloneqq \bigcap_{D \in \Delta} \left\{ v \in \M_{\QQ} \mid \langle \rho(D), v \rangle \ge -m_D \right\}.
\]
Note that if \(X\) is complete, then \(Q^*\) is a polytope, and if \(X\) is \(\QQ\)-Gorenstein Fano, then \(Q^*\) is a translation of the moment polytope of the anticanonical divisor of \(X\). We remark that in~\cite{GorensteinFano}, the polytopes \(Q^*\) corresponding to \(\QQ\)-Gorenstein spherical Fano varieties have been combinatorially characterised, establishing a bijection between this class of rational polytopes, and the spherical embeddings \(G/H\hookrightarrow X\) which are \(\QQ\)-Gorenstein Fano.

We may now define the key object of this paper:

\begin{defin}\noindent\phantomsection\label{def: P}
	Let \(X\) be a spherical variety, then 
	\[
		\widetilde{\wp}(X) \coloneq \dim X-\rk X-\sup \left\{ \sum_{D\in\Delta}\left(m_D-1+\big\langle\rho(D),\vartheta\big\rangle\right)\,\big|\,\vartheta\in Q^*\cap\T \right\},
	\]
	where \(\T \coloneq -\V^\vee\) is the negative of the dual cone of \(\V\).
\end{defin}

\begin{rem}
This definition relates to the definition of \(\wp(X)\) in~\cite{GH15} in the following way: \(\widetilde{\wp}(X)=\dim X-\rk X-\wp(X)\).
\end{rem}

\section{The generalised Mukai conjecture for spherical varieties}\noindent\phantomsection\label{sec:mukai proof}
The goal of this section is to prove \Cref{thm:key theorem}. This provides a proof of \Cref{thm:Generalised Mukai conjecture}, the generalised Mukai conjecture for \(\QQ\)-factorial spherical varieties.
As mentioned in the introduction, the strategy is to interpret the function \(\widetilde{\wp}\) in terms of absolute complexities, reducing the proof of \Cref{thm:Generalised Mukai conjecture} to \Cref{complexity zero means toric}. 

\subsection{Interpreting the \texorpdfstring{\(\widetilde{\wp}\)}{p-bar} function}
Let \(X\) be a complete \(\QQ\)-Gorenstein spherical variety.
In this section, we use the notation of the previous introductory sections.

\begin{prop}\noindent\phantomsection\label{polytope linear system correspondence}
	The rational points of \(Q^*\) correspond to the effective \(B\)-invariant \(\QQ\)-divisors which are \(\QQ\)-linearly equivalent to \(-K_X\), i.e.\ divisors \(0\le E=\sum_{D \in \Delta} a_D D\) with \(a_D \in \QQ\), such that \(E \sim_{\QQ} -K_X\).
\end{prop}
\begin{proof}
	Fix a basis \(\{\lambda_i\,|\,i=1,\dots,\rk X\}\) of \(\M\), and let \(f_{\lambda_i}\in{\CC(G/H)}^{(B)}\) be the rational function of \(B\)-weight \(\lambda_i\), which is unique up to scaling by constants.
    Since \(X\) is \(\QQ\)-Gorenstein, there is an integer \(m>0\) such that \(-mK_X\) is Cartier.
    Then by our assumption in \Cref{sec:background}, \(G=G^\mathrm{ss}\times S\) is factorial, and \(\OO_X(-mK_X)\) is \(G\)-linearisable.
    It follows that all powers of \(\OO_X(-mK_X)\), i.e.\ \(\OO_X(-rmK_X)\) for \(r \in \ZZ\), inherit a natural \(G\)-linearisation from \(\OO_X(-mK_X)\).
    Since \(\OO_X(-mK_X)\) is \(G\)-linearised, \(H^0(X,\OO_X(-mK_X))\) admits a \(G\)-module structure.
    Let \(\delta \in H^0(X,\OO_X(-mK_X))\) be such that \(\mathrm{div}(\delta) = -mK_X\) (such a section is uniquely defined up to multiplication by an invertible function).
    Since \(-mK_X\) is \(B\)-invariant, the section \(\delta\) is \(B\)-semi-invariant of weight \(m\kappa \in \mathfrak{X}(B)\).
    Let \(u,v \in \mathbb{Z}_{\ge0}\).
	Notice that our choice of \(G\)-linearisations guarantees that if \(s \in H^0(X,{\OO_X(-umK_X)}^{(B)})\) is a \(B\)-semi-invariant global section of weight \(\chi \in \mathfrak{X}(B)\), then \(s^v \in H^0(X,{\OO_X(-uvmK_X)}^{(B)})\) is also \(B\)-semi-invariant of weight \(v\chi \in \mathfrak{X}(B)\).
	
	Each integral point \(\lambda\coloneqq \sum_i{a_i}\lambda_i\in mQ^*\cap\M\) gives rise to the \(B\)-semi-invariant section \(s_{\lambda}\coloneqq \delta\cdot\prod_i f_{\lambda_i}^{a_i}\in {H^0(X,\OO_X(-mK_X))}_{\lambda+m\kappa}^{(B)}\) of weight \(\lambda + m \kappa\) (see~{\cite[Section~9]{GorensteinFano}} or {\cite[Section~17]{Timashev}}).
    For each rational point \(x\in mQ^*\cap\M_{\QQ}\), there is a non-unique \(r_x\in\ZZ_{\ge0}\) such that \(\lambda\coloneqq r_x x\in r_x m Q^* \cap \M\).
	Let \(f_{\lambda}\in{\CC(G/H)}^{(B)}\) have weight \(\lambda\).
    Then \(\delta^{r_x} \cdot f_\lambda\) corresponds to a \(B\)-semi-invariant global section of \(\OO_X(-r_x m K_X)\), to which we associate the \(\QQ\)-divisor \(\frac1{r_x m}\mathrm{div}(\delta^{r_x}\cdot f_\lambda)\sim_\QQ -K_X\).
    It is straightforward to show that this divisor is unique, in particular, that it does not depend on \(r_x\).\end{proof} 

\begin{rem}\noindent\phantomsection\label{Q divisor proposition}
    From the proof of \Cref{polytope linear system correspondence}, it follows that the effective \(B\)-invariant \(\QQ\)-divisors \(E\) with \(E\sim_\QQ-K_X\) are those of the form \(E=-K_X+\frac{1}{m}\divv(f_{m\vartheta})\), for a rational point \(\vartheta\in Q^*\) and \(m\in\ZZ_{>0}\) such that \(m\vartheta\in mQ^*\cap\M\) is integral.

\end{rem}

\begin{rem}\noindent\phantomsection\label{P function meaning}
As \(X\) is complete, \(Q^*\) is compact, so the supremum in \(\widetilde{\wp}\) is a maximum. Since \(\rk \Cl(X)=\#\Delta-\rk X\), we may rewrite \(\widetilde{\wp}(X)\) as follows:
    \[
        \widetilde{\wp}(X)=\dim X+\rk\Cl(X)-\max\left\{\sum_{D\in\Delta}\left(m_D+\langle\rho(D),\vartheta\rangle\right)\,\big|\,\vartheta\in Q^*\cap\T\right\}.
    \]

    The second term in the equation above chooses a rational point \(\vartheta\in Q^*\cap\T\) corresponding to the \(B\)-invariant principal \(\QQ\)-divisor \(E'\coloneq\frac{1}{m}\mathrm{div}(f_{m\vartheta})\) with \(m \in \ZZ_{>0}\) such that \(m\vartheta\in mQ^*\cap\T\cap\M\), which attains the maximum value of \(d(E')\).
    Indeed, from the description of the class group of a spherical variety, by fixing a basis \(\{\lambda_1,\dots,\lambda_r\}\) of \(\M\), for \(m\vartheta=\vartheta_1 \lambda_1 + \cdots + \vartheta_r \lambda_r\in\M\) with \(\vartheta_i \in \ZZ\), we have:
    \begin{align*}
        d\left(\frac{1}{m}\mathrm{div}(f_{m\vartheta})\right) &= \frac{1}{m}d \left( \mathrm{div}\left( f_1^{\vartheta_1} \cdots f_r^{\vartheta_r}\right)\right) = \frac{1}{m}d\left(\sum_{i=1}^r\vartheta_i\mathrm{div}(f_{\lambda_i})\right)\\
        &=\frac{1}{m}\sum_{i=1}^r \vartheta_i\sum_{D\in\Delta}\langle\rho(D),\lambda_i\rangle=\langle\sum_{D\in\Delta}\rho(D),\vartheta\rangle.
    \end{align*}

    So the second term computes the maximum value of \(d(E)\) where \(E=-K_X+E'\) is an effective \(B\)-invariant divisor \(E'\sim_\QQ-K_X\), where \(E'\) is as above. Therefore
    \[
        \widetilde{\wp}(X)=\dim X+\rk\Cl(X)-\max\left\{d\big(-K_X+\frac{1}{m}\divv(f_{m\vartheta})\big)\,\big|\,m\in\ZZ_{>0},\,\vartheta\in mQ^*\cap\T\cap\M\right\}.
    \]
    That is, \(\widetilde{\wp}(X)\) is the minimum absolute complexity of a \(B\)-invariant divisor \(E\) in the \(\QQ\)-linear system of \(-K_X\), subject to the constraint \(E=-K_X+\frac{1}{m}\divv(f_{m\vartheta})\) for some \(\vartheta\in mQ^*\cap\T\cap\M\) and \(m\in\ZZ_{>0}\).
\end{rem}

\begin{rem}\noindent\phantomsection\label{struture of divisors from Q}
    Restricting to elements \(\vartheta\in\T\) in the equations above, corresponds to choosing \(\QQ\)-divisors of the form \(E=-K_X+\frac{1}{m}\divv(f_{m\vartheta})\) with \(\mathrm{ord}_Y(f_{m\vartheta})\le 0\) for every \(G\)-invariant prime divisor \(Y\) lying \(G\)-equivariantly over \(G/H\).
    That is, for every \(G\)-equivariant birational morphism \(\pi \colon X' \to X\) such that \(Y\subseteq X'\) is a \(G\)-invariant prime divisor, we have \(\mathrm{ord}_Y(f_{m\vartheta})\le0\). In particular, each \(G\)-invariant prime divisor in \(E\) has coefficient in \([0,1]\).
\end{rem}

\subsection{Constructing a log canonical pair}\noindent\phantomsection\label{sec:constructing-lc-pair}
Let \(X\) be a complete \(\QQ\)-Gorenstein spherical variety.
Let \(D=-K_X+\frac{1}{m}\divv(f_{m\vartheta})\) be a divisor arising from \(\vartheta\in\M_{\QQ}\) with \(m\vartheta\in mQ^*\cap\T\cap \M\) for \(m\in\ZZ_{>0}\) in the sense of \Cref{Q divisor proposition}.
In this section, we replace \(D\) by some \(\widetilde{D}\sim_\QQ D\) so that \((X,\widetilde{D})\) is a log canonical pair.
For this, we use \Cref{SNC machine} to build a simple normal crossings divisor.
Recall that a divisor \(Y = \sum Y_i\) on \(X\) has \emph{simple normal crossings} (or \(Y\) is an \emph{snc divisor}) if each irreducible summand \(Y_i\) is smooth, and if \(Y\) is defined in a neighbourhood of any point by an equation in local analytic coordinates of the type \(z_1 \cdots z_k = 0\) for some \(k \le \dim(X)\).
A \(\QQ\)-divisor \(\sum a_i Y_i\) has \emph{simple normal crossing support} if \(\sum Y_i\) is an snc divisor.

\begin{lemma}[{\cite[Lemma~9.1.9]{Lazarsfeld}}]\noindent\phantomsection\label{SNC machine}
    Let $Y$ be a simple normal crossings divisor on a smooth variety \(X\), and let $|V|$ be a basepoint free linear system on \(X\).
    Then a general divisor \(A\in|V|\) makes \(Y+A\) a simple normal crossings divisor on \(X\).
\end{lemma}

\begin{prop}\noindent\phantomsection\label{gorenstein lc pair}
    There is a log canonical pair \((X,\widetilde{D})\) with \(\widetilde{D}\sim_\QQ D\) and \(d(\widetilde{D})=d(D)\).
\end{prop}
\begin{proof}
    Take a \(G\)-equivariant toroidal resolution \(\pi \colon X' \to X\) given by intersecting each colored cone of the colored fan of \(X\) with the valuation cone, then discolouring every colored cone, and refining any resulting non-smooth cones into smooth cones via repeated stellar subdivision.
    Let \(\mathrm{Exc}(\pi)=\sum_i E_i\) be the exceptional locus.
    Then the \(G\)-invariant spherical boundary divisor \(\sum_i E_i+\sum_{X_j\in\Delta\setminus \D} \pi_{\ast}^{-1}(X_j)\) of \(X'\) has simple normal crossings (\(\pi_{\ast}^{-1}(\cdot)\) denotes the strict transform). 
    
    Using the notation from above, we may write \(D=-K_X+\frac{1}{m}\divv(f_{m\vartheta})=\sum_i a_i X_i+\sum_j b_j D_j\) (here, \(X_i \in \Delta \setminus \D\) runs through the \(G\)-invariant boundary divisors of \(X\) and \(D_j \in \D\) runs through the colors on \(X\)).
    Let \(X'_i = \pi_{\ast}^{-1}(X_i)\) and \(D_j' = \pi_{\ast}^{-1}(D_j)\) be the strict transforms of \(X_i \in \Delta \setminus \D\) and \(D_j \in \D\) respectively.
    Notice that the \(D'_j\) are the colors of \(X'\).
    Let us consider the strict transform of \(D\), i.e.\ \(D'\coloneq \pi_{\ast}^{-1}(D) = \sum_i a_i X_i' + \sum_j b_j D_j'\).
    We replace \(D'\) with the following \(\QQ\)-divisor: 
    \begin{equation}\noindent\phantomsection\label{eq:D-prime-tilde}
        \widetilde{D}'=\sum_i a_i X'_i+\sum_j \left(\left(\sum_{k=1}^{\lfloor b_j\rfloor} S'_{j,k}\right)+\{b_j\}S'_{j,\lceil b_j\rceil}\right),
    \end{equation}
    where \(S'_{j,k}\in |D'_j|\) is a general member of the linear system of the color \(D'_j\); \(\lfloor\cdot\rfloor,\lceil\cdot\rceil\colon \QQ \to \ZZ\) denote the floor and ceiling functions; and \(\{x\} = x - \lfloor x \rfloor\) is the fractional part of \(x \in \QQ\).

    We set \(\widetilde{D} \coloneqq \overline{\pi(\widetilde{D}'|_{X'\setminus\mathrm{Exc}(\pi)})}\subseteq X\), where the closure is taken inside \(X\).
    To be more precise, by this construction we mean
    \[
        \widetilde{D} = \sum_i a_i \overline{\pi\left(X'_i|_{X'_0}\right)} + \sum_j \left(\left(\sum_{k=1}^{\lfloor b_j \rfloor} \overline{\pi\left(S'_{j,k}|_{X'_0}\right)} \right) + \{b_j\} \overline{\pi\left(S'_{j,\lceil b_j\rceil}|_{X'_0}\right)}\right),
    \]
    where we have set \(X'_0 \coloneqq X'\setminus\mathrm{Exc}(\pi)\).
    Notice that \(\overline{\pi(X'_i|_{X'_0})} = X_i\) for all \(X_i \in \Delta \setminus \D\).
    
    We claim that \(\widetilde{D}\sim_\QQ D\).
    Indeed, for each \(j,k\) there exists \(g_{j,k} \in \CC(G/H)\) such that \(S'_{j,k} - D'_j = \divv(g_{j,k})\).
    Notice that not all coefficients of the divisors \(S'_{j,k}\) in \(\widetilde{D}'\) are integral.
    We can clear denominators by multiplying with an appropriate multiple \(m' \in \ZZ_{\ge0}\), that is
    \[
        m' \widetilde{D}'=\sum_i m'a_i X'_i+\sum_j \left(\left(\sum_{k=1}^{\lfloor b_j\rfloor} m'S'_{j,k}\right) + m' \{b_j\}S'_{j,\lceil b_j\rceil}\right),
    \]
    where now each coefficient is integral.
    Set \(m_{j,k} \coloneqq m'\) except for \(k=\lceil b_j \rceil\) where we set \(m_{j,\lceil b_j \rceil} \coloneqq m'\{b_j\}\).
    Let \(g \coloneqq \prod_{j,k} g_{j,k}^{m_{j,k}} \in \CC(G/H)\).
    By construction, it is clear that \(m'(\widetilde{D}' - D') = \divv(g)\).
    We get \((m'(\widetilde{D}' - D'))|_{X'\setminus\mathrm{Exc}(\pi)} = \divv(g)|_{X'\setminus\mathrm{Exc}(\pi)}\), and thus \(m'(\widetilde{D} - D) = \divv(g)\) on \(X\) as \(\pi(\mathrm{Exc}(\pi))\) has codimension at least \(2\) in \(X\).
    
    We claim that this modification provides a log canonical pair \((X,\widetilde{D})\) satisfying \(\widetilde{D}\sim_\QQ D\) and \(d(\widetilde{D})=d(D)\).
    Note that it only remains to verify the log canonical property.
    As \(X'\) is toroidal, by~\cite[Proposition~2.6]{BrionBasepointFree}, the colors \(D_j'\) of \(X'\) are basepoint free and \(\sum_i X_i' + \sum_i E_i\) is an snc divisor.
    Therefore by iteratively applying \Cref{SNC machine}, \(\pi_{\ast}^{-1}(\widetilde{D})+\sum_i E_i\) is a \(\QQ\)-divisor with simple normal crossing support, and thus \(\pi \colon X' \to X\) is a log resolution for the pair \((X, \widetilde{D})\).
    Moreover, by \Cref{struture of divisors from Q}, each component of \(\pi_{\ast}^{-1}(\widetilde{D})\) has coefficient in \((0,1]\).

    From the construction of \(\widetilde{D}\), there exists some \(m' \in \ZZ_{>0}\) and \(g \in \CC(G/H)\) such that 
    \[
        \widetilde{D} = D + \frac1{m'}\divv(g) = -K_X + \frac1m\divv(f_\vartheta) + \frac1{m'}\divv(g).
    \]
    Notice that \(g\) has no poles nor zeros along each \(G\)-invariant divisor in \(X'\). 
    Indeed, as the \(D_j'\) are basepoint free, none of the general members \(S'_{j,k} \in |D'_j|\) are contained in the finite set of \(G\)-invariant divisors in \(X'\).
    Therefore \(\divv(g) = m'(\widetilde{D} - D)\) has no \(G\)-invariant components.
    
    Finally, we bound the discrepancies of the exceptional locus, which are the coefficients \(c_i\) of the prime divisors \(E_i\subseteq\mathrm{Exc}(\pi)\) in \(K_{X'}=\pi^*(K_X+\widetilde{D})-\pi_{\ast}^{-1}(\widetilde{D})+\sum_i c_i E_i\).
    Multiplying this equation by \(mm'\) yields:
    \begin{gather*}
        mm'K_{X'}=\pi^*(m'\divv(f_{m\vartheta})+m\divv(g))-\pi_{\ast}^{-1}(mm'\widetilde{D})+mm'\sum_i c_i E_i.
    \end{gather*}
    Computing the coefficient of \(E_i\) in this equation and using that \(\vartheta\in\T\) gives:
    \begin{align*}
        -mm'&=m'\ord_{E_i}(f_{m\vartheta})+m\ord_{E_i}(g)+mm'c_i,\\
        c_i &= -1 - \frac{1}{m} \ord_{E_i}(f_{m\vartheta}),\\
        c_i &= -1 - \langle \rho(E_i), \vartheta \rangle \ge-1.
    \end{align*}
     
    Therefore \((X,\widetilde{D})\) is a log canonical pair.
\end{proof}

\begin{thm}\noindent\phantomsection\label{P function theorem}
	Let \(X\) be a complete \(\QQ\)-Gorenstein spherical variety, then \[\widetilde{\wp}(X)\ge\gamma(X),\] with \(\widetilde{\wp}(X)<1\) only if \(X\) is isomorphic to a toric variety.
\end{thm}

\begin{proof}
    By \Cref{P function meaning}, \(\widetilde{\wp}(X)=\dim X+\rk\Cl(X)-d(D)\) where \(D\) is some \(B\)-invariant \(\QQ\)-divisor with \(D\sim_\QQ-K_X\).
    Then by \Cref{gorenstein lc pair}, we may replace \(D\) with \(\widetilde{D}\), so that \((X,\widetilde{D})\) is a log canonical pair with \(\widetilde{D}\sim_\QQ -K_X\) and \(d(\widetilde{D})=d(D)\).
    Therefore 
    \[
        \widetilde{\wp}(X)=\dim X+\rk\Cl(X)-d(D)=\dim X+\rk\Cl(X)-d(\widetilde{D})=\gamma(X,\widetilde{D})\ge\gamma(X).
    \]
    Suppose that \(\widetilde{\wp}(X)<1\), then \(\gamma(X)<1\), so there is a log canonical pair \((X,\widetilde{D})\) with \(K_X+\widetilde{D}\sim_\QQ0\) such that \(\gamma(X,\widetilde{D})<1\).
    Therefore by \Cref{complexity zero means toric}, \(X\) is isomorphic to a toric variety.
\end{proof}

\subsection{Proof of the spherical generalised Mukai conjecture}\noindent\phantomsection\label{sec:proof-of-sph-mukai}
We have now all the tools to prove the generalised Mukai conjecture for \(\QQ\)-factorial spherical varieties.
\begin{proof}[Proof of \Cref{thm:Generalised Mukai conjecture}]
    By \Cref{P function conjecture intro} and \Cref{P function theorem},
    \begin{equation}\noindent\phantomsection\label{eq:seq-ineqs-Mukai}
		(\iota_X-1)\rho_X\le\dim X-\widetilde{\wp}(X)\le \dim X-\gamma(X)\le \dim X,
    \end{equation}
    where in the last inequality we have used \(\gamma(X)\ge0\) (see \Cref{complexity zero means toric}).

    Next, let us address the equality case in \Cref{thm:Generalised Mukai conjecture}.
    Notice that the sequence of inequalities in~\eqref{eq:seq-ineqs-Mukai} implies that
    \[
        (\iota_X-1)\rho_X=\dim X\quad \Rightarrow \quad \widetilde{\wp}(X) = 0 \quad \Rightarrow \quad \gamma(X) = 0.
    \]
    Therefore, by \Cref{P function theorem}, if \((\iota_X-1)\rho_X = \dim{X}\), then \(X\) is isomorphic to a toric variety.
	Finally, from the proof of \Cref{thm:Generalised Mukai conjecture} for \(\QQ\)-factorial toric varieties~\cite{Fujita2018TheGM}, there is the equality \((\iota_X-1)\rho_X=\dim X\) if and only if \(X\cong {(\PP^{\iota_X-1})}^{\rho_X}\).
\end{proof}

\subsection{Some remarks on the Mukai conjecture}

For a \(\QQ\)-factorial Fano variety \(X\), we have observed that the following inequality holds when \(X\) is spherical, which is a stronger bound than in \Cref{Mukai conjecture}. One might suspect that such an inequality holds for non-spherical Fano varieties.

\begin{prop}\noindent\phantomsection\label{reformulation}
    Let \(X\) be a \(\QQ\)-factorial spherical Fano variety, then
    \[
        (\iota_X-1)\rho_X+\gamma(X)\le \dim X.
    \]

\end{prop}
This inequality implies \Cref{Mukai conjecture} for spherical varieties as we saw in the proof of \Cref{thm:Generalised Mukai conjecture} in \Cref{sec:proof-of-sph-mukai}.

 We thank Kento Fujita for suggesting the following example, which demonstrates that the inequality of \Cref{reformulation} does not hold in general.

\begin{example}\noindent\phantomsection\label{ex:delPezzo}
    Let \(X=\mathrm{Bl}_8\PP^2\) be a del Pezzo surface of degree 1, and let \((X,D)\) be a log canonical pair with \(D\coloneq\sum a_i D_i\sim_\QQ -K_X\), where the \(D_i\) denote the irreducible components of \(D\) and the \(a_i\) are rational numbers in \((0,1]\).
    Then
    \[
        1=(-K_X\cdot \sum a_i D_i)=\sum a_i(-K_X\cdot D_i)\geq \sum a_i=d(D),
    \]
    where we have used \((-K_X \cdot D_i) >0\), and thus \((-K_X\cdot D_i)\ge1\) by Kleiman's criterion of ampleness.
    From this, we get the following bound on \(\gamma(X)\):
    \[
        \gamma(X) = \inf\{ \dim(X) + \rho_X - d(D) \mid (X,D)\text{ log canonical pair with }D\sim_\QQ -K_X\} \ge 10.
    \]
    Hence \((\iota_X-1)\rho_X+\gamma(X)\ge10\), but \(\dim X =2\).
\end{example}

\begin{rem}
It may be interesting to consider this question with the \emph{complexity} instead of the absolute complexity (see \cite[Definition~1.1]{Geomcharoftoric}). Let \((X,D)\) be a log pair. A \emph{decomposition} of \((X,D)\) is \(\sum_{i=1}^k a_i S_i\le D\), where \(a_i\ge0\), and \(S_i\ge0\) are \(\ZZ\)-divisors. Let \(r\) be the rank of the vector space spanned by \(\{S_1,\dots,S_r\}\) in the Néron-Severi group.
The \emph{complexity} of this decomposition is \(\dim X+r-\sum_{i=1}^k a_i\), and the \emph{complexity} \(c(X,D)\) is the minimum complexity over all decompositions of \(D\).

We define:
\[
0\le c(X)\coloneq\min\{c(X,D)\mid(X,D)\text{ log Calabi-Yau pair}\}\le\gamma(X),
\]
which satisfies \(c(X)<1\) if and only if \(X\) is isomorphic to a toric variety. 

We may consider the following question.
\begin{question}\label{complexity mukai question} Let \(X\) be a smooth Fano variety, then does the following hold?
\[
(\iota_X-1)\rho_X+c(X)\le\dim X.
\]

\end{question}
It is clear that \Cref{complexity mukai question} implies the generalised Mukai conjecture. Moreover, it is easy to verify \Cref{complexity mukai question} in some special cases. In particular, for \(\rho_X=1\), for \(\iota_X=1\), for \(\dim X\le3\), and for any \(X\) such that \(\mathrm{Nef}(X)\subset N^1(X)\) is a smooth cone with respect to \(\Pic(X)\), provided that the Ambro-Kawamata effective non-vanishing conjecture holds for \(X\) (see \cite{Nonvanishing} and Proposition 4.4 within for details on this conjecture.)
\end{rem}

\begin{rem} One might ask if \Cref{thm:Generalised Mukai conjecture} holds when \(X\) is not \(\QQ\)-factorial. In this setting \(\Pic(X)_\QQ\neq\Cl(X)_\QQ\), and this question should be posed with \(\rho_X\coloneq\rk\Pic(X)\). Indeed, let \(X\) be the non-\(\QQ\)-factorial toric Fano threefold given by the face fan of the reflexive polytope in Figure~\ref{fig:non-Q-factorial}.

\begin{figure}[!ht]
\tdplotsetmaincoords{55}{9}
\begin{tikzpicture}[tdplot_main_coords, scale=1.2]

\coordinate (O) at (0,0,0);
\coordinate (A) at (-1,-1,-1);
\coordinate (B) at (0,-1,-1);
\coordinate (C) at (1,0,-1);
\coordinate (D) at (1,1,-1);
\coordinate (E) at (0,1,-1);
\coordinate (F) at (-1,0,-1);
\coordinate (G) at (0,0,1);
\coordinate (label) at (-1.3,1,-1.16);
\coordinate (label2) at (-1.3,1,-1.05);

\fill (0,0,0) circle (2pt);

\draw[thick]  (A) -- (B);
\draw[thick]  (B) -- (C);
\draw[thick]  (C) -- (D);
\draw[thick,dashed]  (D) -- (E);
\draw[thick,dashed]  (E) -- (F);
\draw[thick]  (A) -- (F);

\draw[thick] (G) -- (A);
\draw[thick] (G) -- (B);
\draw[thick] (G) -- (C);
\draw[thick] (G) -- (D);
\draw[thick,dashed] (G) -- (E);
\draw[thick] (G) -- (F);
\draw[thin] (E) -- (label);
\draw[thin] (label) -- (label2);

\node[anchor=north east] at (A) {\scriptsize \((-1,-1,-1)\)};
\node[anchor=north] at (B) {\scriptsize \((0,-1,-1)\)};
\node[anchor=west] at (C) {\scriptsize \((1,0,-1)\)};
\node[anchor=south west] at (D) {\scriptsize \((1,1,-1)\)};
\node[anchor=south] at (label2) {\scriptsize \((0,1,-1)\)};
\node[anchor=east] at (F) {\scriptsize \((-1,0,-1)\)};
\node[anchor=south] at (G) {\scriptsize \((0,0,1)\)};

\end{tikzpicture}
\caption{Reflexive polytope yielding a non-\(\QQ\)-factorial toric Fano threefold.\label{fig:non-Q-factorial}}
\end{figure}

Using the toric dictionary one can compute \(\rk\Cl(X)=4\), \(\rho_X=1\), and \(\iota_X=2\). So
\[
(\iota_X-1)\rk\Cl(X)=4>\dim X,\quad \text{yet}\quad (\iota_X-1)\rho_X=1<\dim X.
\]
Moreover, in \cite{Fujita2018TheGM} (which takes \(\rho_X=\rk\Pic(X)\)), \Cref{thm:Generalised Mukai conjecture} is shown for \(\QQ\)-Gorenstein toric Fano varieties.
\end{rem}


\section{An example}\noindent\phantomsection\label{sec:examples}In this section we illustrate the paper with an example.

\begin{example}
	 We consider a classical spherical variety: the variety of smooth conics \(G/H=\mathrm{SL}_3/Z(\mathrm{SL}_3)\mathrm{SO}_3\) (cf.~\cite[Table~A]{Wasserman}).
	We have \(\M=\ZZ(2\alpha_1)\oplus\ZZ(2\alpha_2)\), with two colors \(D_1\) and \(D_2\) of type 2a, with spherical roots \(\Sigma=\{2\alpha_1, 2\alpha_2\}\), and dual vectors \(\rho(D_1)=(2,-1)\), \(\rho(D_2)=(-1,2)\in\N\).
	The simple embedding \(X\) with colored cone \((\mathrm{cone}((-1,0),(2,-1)),\{D_1\})\) is the parameter space \(\PP^5\) of all conics, with a unique \(G\)-invariant divisor \(X_1\) parametrising the singular conics, and \(\Delta=\{X_1,D_1,D_2\}\). 
	The polytope \(Q^*\) is the triangle \(\mathrm{conv}((-1,-1),(1,0),(1,3))\), and \(\mathcal{T}=\mathrm{cone}(\Sigma)\) is the positive orthant of \(\M_\QQ\), see Figure~\ref{conics picture}.

	\begin{figure}[ht!]
		\centering
		\scalebox{1}{
		\begin{tikzpicture}
			\node[] at (0,2.5) {$\mathcal{N}_\QQ$};
			\draw[->] (-1.6,0) -- (1.6,0);
			\draw[->] (0,-1.6) -- (0,1.6);
			\draw[-] (0,0) -- (-1/2,2/2);
			\draw[-] (0,0) -- (2/2,-1/2);
			\node[] at (0,1.8) {$\frac{1}{2}\varpi_1^\vee$};
			\node[] at (2,0) {$\frac{1}{2}\varpi_2^\vee$};
			\fill[black] (-1/2,2/2) circle (1.5pt) node[label={[label distance=-0.1cm]}]{};
			\node[] at (-1.05,1) {$\rho(D_2)$};
			\node[] at (1,-0.8) {$\rho(D_1)$};

			\fill[black] (2/2,-1/2) circle (1.5pt) node[label={[label distance=-0.8cm]}]{};
			\coordinate (a) at (-0.007,-0.007);
			\coordinate (b) at (-1.5,-0.007);
			\coordinate (c) at (-0.007,-1.5);

			\fill[fill=gray!50!white] (a) -- (b) -- (c)  -- cycle;
			\node[] at (-0.43,-0.43) {$\V$};
		\end{tikzpicture}
		\quad \begin{tikzpicture}
			\coordinate (a) at (0,0);
			\coordinate (b) at (-1.5,0);
			\coordinate (c) at (1.5,-1.5/2);
			\fill[fill=gray!50!white] (a) -- (b) -- (c);

			\node[] at (0,2.5) {$\PP^5\hookleftarrow G/H$};
			\draw[->] (-1.6,0) -- (1.6,0);
			\draw[->] (0,-1.6) -- (0,1.6);
			\draw[-] (0,0) -- (1.5,-1.5/2);
			\fill[black] (2/2,-1/2) circle (1.5pt) node[label={[label distance=-0.8cm]}]{};
			\node[] at (0,1.9) {};
			\node[] at (2.1,0) {};
		\quad\quad \end{tikzpicture}

		\begin{tikzpicture}
			\coordinate (a) at (-0.01,-0.01);
			\coordinate (b) at (1/1.5-0.004,0-0.01);
			\coordinate (c) at (1/1.5+0.007,3/1.5-0.1);
			\coordinate (d) at (0-0.01,3/1.5-0.1);
			\fill[fill=gray!50!white] (a) -- (b) -- (c) -- (d) ;
            \draw[->] (-1.6,0) -- (1.6,0);
			\draw[->] (0,-1.6) -- (0,3/1.5);
			  \path[draw] (-1/1.7,-1/1.7) node[left] (5) {$(-1,-1)$}
					-- (1/1.7, 3/1.7) node[right] (2) {$(1,3)$}
					-- (1/1.7, 0) node[above right] (4) {$(1,0)$}
					-- cycle;
					\node[] at (0.1,2.5) {$Q^*\subset\M_\QQ$};
		\end{tikzpicture}}
		\caption{\centering The valuation cone and colors of \(G/H\) in \(\N_\QQ\); the colored fan of \(\PP^5\hookleftarrow G/H\); and the polytope \(Q^*\) with \(\T\) shaded.\label{conics picture}}
	\end{figure}
	Using~\cite[Theorem~4.2]{BrionAnticanonical}, \(-K_X=X_1+D_1+D_2\) and \(d(-K_X)=\sum_{D\in\Delta}m_D=3\).
    To illustrate \Cref{P function meaning}, the relations in \(\Pic(X)\) are \([0]=[\mathrm{div}(f_{2\alpha_1})]=[2D_1-D_2-X_1]\) and \([0]=[\divv(f_{2\alpha_2})]=[2D_2-D_1]\), so for \(\vartheta=x(2\alpha_1)+y(2\alpha_2)\) we get:
    \[
        d(x\mathrm{div}(f_{2\alpha_1})+y\mathrm{div}(f_{2\alpha_2}))=d(D_1(2x-y)+D_2(2y-x)+X_1(-x))=y=\langle\sum_{D\in \Delta}\rho(D),\vartheta\rangle.
	\]
   We compute
    \begin{align*}
        \widetilde{\wp}(X)&=\max\left\{\dim X+\rho_X-\sum_{D\in\Delta}m_D-y\,\big|\,(x,y)\in Q^*\cap\T\right\},\\
        &=\max\left\{3+y\,|\,(x,y)\in Q^*\cap\T\right\},\\
        &=0.
    \end{align*}
    We observe \(\widetilde{\wp}(X)<1\), detecting that \(X\) is isomorphic to a toric variety.
    The maximum of the linear program is achieved at the vertex \((1,3)\) of \(Q^*\) corresponding to the divisor \(6D_2\), where \(D_2\) is a \(B\)-invariant hyperplane in \(\PP^5\).  The log canonical pair of \Cref{gorenstein lc pair} is \((X,H_1+\dots+H_6)\), where the \(H_i\in|D_2|\) are general hyperplanes in \(\PP^5\). 
\end{example}

\section{A note on the Mukai-type conjecture}\label{sec:Mukai type}

In 2023, extending the goal of the Mukai conjecture, Gongyo posed \Cref{Mukai-type conjecture}, which conjecturally characterises products of projective spaces of not necessarily the same dimension among smooth Fano varieties. The two results of this section are the answer to a question of Gongyo on this conjecture, and to show that while \Cref{Mukai-type conjecture} does not imply the Mukai conjecture, by a simple argument, \Cref{Mukai-type conjecture} admits the same lower bound as the generalised Mukai conjecture in the setting of spherical varieties.

Very recently, \Cref{Mukai-type conjecture} was proved in full generality \cite{enwright2025characterizationproductsprojectivespaces}, for which the aforementioned bound is not shown to hold in general. 


\begin{defin}\label{total index}
    Let \(X\) be a Fano variety. Let \(\mathbb{K}\in\{\ZZ,\QQ\}\). The \(\mathbb{K}\)-\emph{total index} \(\tau_X(\mathbb{K})\) of \(X\) is
    \[
        \tau_{X}(\mathbb{K})\coloneq\sup\left\{\sum_{i=1}^ka_i\mid \sum_{i=1}^ka_iS_i=-K_X,\,S_i>0\text{ is a nef }\ZZ\text{-divisor},\,a_i\in\mathbb{K}_{>0}\right\}.
    \]
\end{defin}

\begin{conjecture}[{\cite[Conjecture~1.2]{gongyo}}]\label{Mukai-type conjecture} Let \(X\) be a smooth Fano variety, then
\[
\dim X+\rho_X-\tau_X(\mathbb{K})\ge0,
\]
with equality if and only if \(X\cong\PP^{n_1}\times\dots\times\PP^{n_{\rho_X}}\).
\end{conjecture}

In \Cref{Mukai-type proof}, we give a succinct proof of \Cref{Mukai-type conjecture} for spherical Fano varieties. Firstly, we answer \Cref{question} asked by Gongyo, comparing \Cref{Mukai-type conjecture} with \(\ZZ\)-coefficients to \Cref{Mukai-type conjecture} with \(\QQ\)-coefficients.

\begin{question}\label{question} Let \(X\) be a smooth toric Fano variety. Do we have \(\tau_{X}(\ZZ)=\tau_X(\QQ)\)?

\end{question}

The following example shows the answer to \Cref{question} is negative. 

\begin{example} Let \(\pi:X\to\PP^4\) be the Voskresenskij-Klyachko fourfold (also called the del Pezzo variety \(V^4\)). It is the blowup of \(\mathbb{P}^4\) in five general points \(p_1,\dots,p_5\), followed by flipping the strict transform of each of the ten lines through the pairs of points \(p_ip_j\). \(X\) is a smooth toric Fano variety whose fan \(\Sigma\) has rays
\[\Sigma^{(1)}=
\begin{pmatrix}
1&-1&0&0&0&0&0&0&1&-1\\
0&0&1&-1&0&0&0&0&1&-1\\
0&0&0&0&1&-1&0&0&1&-1\\
0&0&0&0&0&0&1&-1&1&-1
\end{pmatrix}.
\]
The toric description of \(X\) makes it easy to compute the nef cone (for example, using \texttt{Macaulay2}). We use the basis \(\Pic(X)=\ZZ\pi^*(H)\oplus\ZZ E_1\oplus\dots\oplus\ZZ E_5\), where \(E_1,\dots,E_5\) are the exceptional loci, and \(H\) is a hyperplane on \(\PP^4\). In this basis, the rays of the nef cone are spanned by the following divisor classes:
\[\mathrm{Nef}(X)^{(1)}=\left\{
\begin{array}{lc}
\pi^*(2H)-\sum_{i=1}^5E_i,\\
\pi^*(2H)-\sum_{i=1}^5E_i-E_\ell\quad(\ell=1,\dots,5),\\
\pi^*(3H)-2\sum_{i=1}^5E_i,\\
\pi^*(3H)-2\sum_{i=1}^5E_i+E_\ell\quad(\ell=1,\dots,5)
\end{array}\right\}.
\]
Clearly the parts \(\{S_i\}\) of any \(\ZZ\)-nef-partition \(\sum a_i S_i=-K_X\) satisfy
\[
S_i\in P_X\coloneq\mathrm{Nef}(X)\cap(-K_X-\mathrm{Nef}(X))\cap\Pic(X)\setminus 0.
\]
In this case, \(P_X=\mathrm{Nef}(X)^{(1)}\cup\{-K_X=\pi^*(5H)-3\sum_{i=1}^5E_i\}\).

Observe that three or more classes in \(P_X\) do not sum to the class of \(-K_X\), and there are six \(\ZZ\)-nef-partitions of \(-K_X\) into two elements of \(P_X\). Therefore \(\tau_X=2\). However, with rational coefficients we can do better:
\begin{align*}
-K_X=\frac{1}{2}\sum_{\ell=1}^5\left(\pi^*(2H)-\sum_{i=1}^5E_i-E_\ell\right).
\end{align*}
So \(\tau_X(\QQ)\ge5/2>\tau_X(\ZZ)\).
\end{example}

\begin{rem} It is easy to see that the supremum in \Cref{total index} is a maximum, by considering the Hilbert basis of \(\Pic(X)\cap\mathrm{Nef}(X)\). Similarly, using this we may compute the total index explicitly in examples.

\end{rem}

\begin{rem}\label{Mukai-type proof} Let \(X\) be a smooth spherical Fano variety, now here we show \Cref{Mukai-type conjecture} holds for \(X\) by showing the following lower bound \(\dim X+\rho_X-\tau_X(\QQ)\ge\gamma(X)\ge0\), which coincides with the lower bound \(\dim X+\rho_X-\iota_X\rho_X\ge\gamma(X)\) for the generalised Mukai conjecture for spherical Fano varieties, obtained in \Cref{thm:key theorem}.

Indeed, let \(X\) be a smooth spherical Fano variety, and let \(\sum_{i=1}^k a_i S_i=-K_X\) with \(a_i\in\QQ_{\ge0}\) and \(S_i\) nef \(\ZZ\)-divisors, be a nef partition for which the value of \(\tau_X(\QQ)\) is attained. Let \(S_{i,j}'\in|S_i|\) be general members of the linear system. Then \(F=\sum_{i=1}^k(\sum_{i=1}^{\lfloor a_i\rfloor}S_{i,j}'+\{a_i\}S_{i,\lceil a_i\rceil}')\) is another nef partition with the same coefficient sum. Since on a complete spherical variety, nef and basepoint free Cartier divisors coincide \cite[Corollary~3.2.11]{PerrinSpherical}, we have \((X,F)\) is a log canonical pair. Therefore \(\dim X+\rho_X-\tau_X(\QQ)\) is the complexity of the log Calabi-Yau pair \((X,F)\), so by \Cref{complexity zero means toric}, this is a non-negative quantity. Equality \(\dim X+\rho_X-\tau_X(\QQ)=0\) holds if and only if \(X\) is a toric variety whose torus invariant prime divisors are the irreducible components of \(F\), and in \(F\) each component has coefficient one, so the classes of the components of \(F\) span the effective cone of \(X\), hence \(\mathrm{Nef}(X)=\mathrm{Eff}(X)\). Now by \cite[Proposition~5.3]{fujino2008smoothprojectivetoricvarieties}, \(X\cong\PP^{n_1}\times\dots\PP^{n_{\rho_X}}\).

\end{rem}

\printbibliography{}

\end{document}